\documentclass[12pt]{amsart}

\usepackage{fullpage}
\usepackage{amsmath}
\usepackage{amsfonts}
\usepackage{amssymb}
\usepackage{amsthm}
\usepackage[utf8]{inputenc}
\usepackage{breqn}
\usepackage{afterpage}
\usepackage{longtable}
\usepackage{indentfirst}
\usepackage{caption}
\usepackage{subcaption}
\usepackage{graphicx}
\graphicspath{ {./images/} }
\newtheorem{theorem}{Theorem}[section]
\newtheorem{lemma}[theorem]{Lemma}
\newtheorem{corollary}[theorem]{Corollary}

\theoremstyle{definition}
\newtheorem{definition}[theorem]{Definition}

\newtheorem{theorem-definition}[theorem]{Theorem-Definition}

\theoremstyle{remark}

\numberwithin{equation}{section}

\newcommand{\C}{\mathbb{C}}
\newcommand{\R}{\mathbb{R}}
\renewcommand{\P}{\mathbb{P}}

\begin{document}

\title{On the support of 
measures of large entropy for polynomial-like maps}
\begin{author}[S.~Bazarbaev]{Sardor Bazarbaev}
    \address{National University of Uzbekistan,  Tashkent, Uzbekistan}
\email{uzedu.bazarbaev@gmail.com}
\end{author}
\begin{author}[F.~Bianchi]{Fabrizio Bianchi}
\address{Dipartimento di Matematica, Università di Pisa, Largo Bruno Pontecorvo 5, 56127 Pisa, Italy}
  \email{fabrizio.bianchi$@$unipi.it}
\end{author}
\begin{author}[K.~Rakhimov]{Karim Rakhimov}

\address{V.I. Romanovskiy Institute of Mathematics of Uzbek Academy of Sciences,  Tashkent, Uzbekistan}
\email{karimjon1705$@$gmail.com }
\end{author}

\begin{abstract}
Let $f$ be a polynomial-like map 
 with dominant
topological degree $d_t\geq 2$
 and let $d_{k-1}<d_t$ be its dynamical degree of order $k-1$.
  We show that the support of 
every ergodic measure whose 
measure-theoretic
entropy is strictly larger than
 $\log \sqrt{d_{k-1} d_t}$
is supported on the Julia set, i.e., the support of the unique measure of maximal entropy $\mu$.
The proof is based on the
exponential speed of convergence
 of the measures
 $d_t^{-n}(f^n)^*\delta_a$ towards
 $\mu$, which is valid for a generic point $a$ and with 
 a controlled
 error bound depending on $a$.
 Our proof also gives a new proof of the same statement in the setting of endomorphisms of $\mathbb P^k(\mathbb C)$
 -- a result due to de Th\'elin and Dinh -- which does not rely on the existence of a Green current.
\end{abstract}
\maketitle

\section{Introduction}
The study of the dynamics of holomorphic
endomorphisms of complex projective spaces $\mathbb P^k:=\mathbb P^k(\mathbb C)$ 
is a central topic
in complex dynamics, see for instance \cite{DS10,Si99} for an overview of the subject. 
Let $f:\mathbb P^k\to \mathbb P^k$ be an endomorphism of algebraic degree $d\ge 2$. There exists a canonical 
positive closed $f^*$-invariant $(1,1)$-current $T$, called the \emph{Green current of $f$},
with the property that the sequence
$d^{-n}(f^n)^*\omega_0$ converges to $T$ for every smooth 
positive closed $(1,1)$-form $\omega_0$ of mass $1$.
The current $T$ has strong geometric properties, in particular, it has H\"older continuous potentials.
As a consequence, the measure $\mu:=T^{\wedge k}$ is well-defined, and it is the unique measure of maximal entropy 
 $k\log d$
 of $f$
\cite{BD01,G03}.
Its support
is called the \emph{Julia set} of $f$. 
By a result of de Th{\'e}lin and Dinh
 \cite{dTh06,Dinh07},
 every
 ergodic measure
 whose
 measure-theoretic entropy is strictly larger than $(k-1)\log d$
 is also supported on the Julia set of $f$. Large classes of examples of such measures are constructed and studied in \cite{BD20,BD22,D12,SUZ,UZ13}.

\medskip

The proof given 
in 
\cite{dTh06,Dinh07}
of the above property crucially relies on the existence 
of the Green current. In particular, it follows from a delicate induction
which makes use of the successive self-intersections $T^{\wedge j}$ of the Green current $T$. It is then unclear how to generalize this result
to more general 
non-algebraic settings, where a 
dynamical Green current does not exist. In this paper, we address
this problem in the case of polynomial-like maps  with dominant topological degree. Our proof will in particular also give a new proof of the result by de Th{\'e}lin and Dinh, which makes no use of the Green current.

\medskip

Recall that polynomial-like maps are proper holomorphic maps $f:U\to V$, where $U \Subset V$ are open subsets of $\C^k$ and $V$ is convex. 
By definition, every polynomial-like map defines a ramified covering $U\to V$
and the \emph{topological degree} $d_t$ of $f$ is well-defined. 
For
every
$0\le p\le k$, one can define 
the
\emph{dynamical degrees}\footnote{Sometimes, see for instance \cite{DS10},
these degrees are denoted by $d^*_p$
to distinguish them from a
different type of dynamical degree
that can also be considered. Since here we will only use one type of dynamical degree, we will use the simpler notation $d_p$.}
    $$    d_p=d_p(f):=\limsup_{n\to\infty}\sup_S\|(f^n)_*(S)\|_U^{1/n},$$
       where
      the supremum is taken over all positive closed $(k-p,k-p)$-currents on $U$ 
       whose mass
       is less than or equal to $1$
       see \cite{DS03,DS10} and Definition \ref{def:degrees} below. 
       Note that we always have $d_0=1$ and $d_k=d_t$. By 
       \cite{BDR23}, the sequence
       $\{d_p\}_{0\leq p\leq k}$ is non-decreasing. Hence, in particular, we have 
      $\max_{0\leq p \leq k-1} d_p = d_{k-1}$. 
       We say that $f$ has  \textit{dominant topological degree}\footnote{In some references,
       maps with $d_{k-1}<d_t$
       are said to have \emph{large}  topological degree, and the name \emph{dominant} is reserved to maps for which
       $\max_{0\leq p \leq k-1} d_p< d_t$. We use here the name \emph{dominant} as, by \cite{BDR23}, these notions are equivalent.}
       if $d_{k-1}< d_t$. Observe that in this case we always have $d_t\ge 2$.

\medskip

Polynomial-like maps 
 with dominant
topological degree
enjoy many of the dynamical properties of endomorphisms
(however, their study is usually technically more involved, because of the lack of a naturally defined Green function).
In particular,
for every such $f$
there 
exists a unique measure $\mu$
of maximal entropy $\log d_t$ 
and
a proper analytic set $\mathcal{E}\subset V$ such that
\begin{equation}
\label{e:intro-conv}
d^{-n}_t(f^n)^*\delta_a
\to \mu
\quad \mbox{
for all }
\quad
a\in  V\setminus \bigcup_{j=0}^\infty f^j(\mathcal{E}),\end{equation}
see
\cite{DS03,DS10}.
Note that, unlike the case of endomorphisms
of $\mathbb P^k$,
here
the set $\mathcal{E}$ 
may not be
$f$-invariant
and hence  $\cup_{j=0}^\infty f^j(\mathcal{E})$ is not, a priori,
an analytic set.

\medskip

The following is our main result.

\begin{theorem}\label{t:intro-nomasspluripolar}
  Let $f:U\to V$ 
be a polynomial-like map 
 with dominant
topological degree.
Every 
   ergodic measure
   $\nu$
   whose measure-theoretic entropy
  satisfies 
 $h_\nu(f)>\log \sqrt{d_{k-1} d_t}$
  is supported on the Julia set $J$.
 \end{theorem}

The
proof of Theorem \ref{t:intro-nomasspluripolar}
consists of a 
quantified version of the classical estimate
by Gromov 
\cite{G03}
for the topological entropy in terms of
the volume
growth of suitable analytic sets in the space of orbits.
It is given in Section \ref{s:4} and
 exploits in a crucial way 
 an
explicit 
exponential rate of
 the
convergence 
\eqref{e:intro-conv},
which we 
 discuss
in Section \ref{s:1}.
In the same section, we
give a more precise 
bound
$\log \beta(f)$
for the entropy in Theorem \ref{t:intro-nomasspluripolar}, see also Theorem \ref{t:logp-nomasspluripolar}.

\medskip

 In the case of 
endomorphisms of $\mathbb P^k$,
we have $\beta (f)= d^{k-1}$, hence the bound $\log \sqrt{d_{k-1} d_t}$ can be improved to $\log d_{k-1} = (k-1)\log d$.
In particular, Theorem \ref{t:intro-nomasspluripolar} gives an alternative proof of the result by de Th{\'e}lin and Dinh
mentioned above \cite{dTh06,Dinh07}.

\begin{corollary}
Let $f$ be an endomorphism of $\mathbb P^k$ of algebraic degree $d\geq 2$. Every ergodic measure whose measure-theoretic entropy is strictly larger than $(k-1)\log d$ is supported on the Julia set.
\end{corollary}

 We refer to Section \ref{s:5} for
 further applications and corollaries of Theorem \ref{t:intro-nomasspluripolar}.

\subsection*{Acknowledgments}
The authors would like to thank Tien-Cuong Dinh for useful discussions about
Section \ref{ss:speed-easy}.
This project has received funding from
 the
 Programme
 Investissement d'Avenir
(ANR QuaSiDy /ANR-21-CE40-0016,
ANR PADAWAN /ANR-21-CE40-0012-01),
from 
the   government of Uzbekistan
 through the grant
 IL-5421101746,
 from the MIUR Excellence Department Project awarded to the Department of Mathematics of the University of Pisa, CUP I57G22000700001,
 and from the PRIN 2022 project MUR 2022AP8HZ9{\_}001.
The second
author is affiliated to the GNSAGA group of INdAM.

\section{Preliminaries}\label{s:1}

\subsection{Polynomial-like maps}
A \emph{polynomial-like map} is a proper holomorphic map $f:U\to V$, where $U \Subset V$ are open subsets of $\C^k$ and $V$ is convex. Homogeneous lifts to $\mathbb{C}^{k+1}$ of endomorphisms of $\mathbb{P}^k$ give examples of
polynomial-like maps.
In dimension $k=1$, any polynomial-like map is conjugate to an actual polynomial on the Julia set \cite{DH85}. However, in higher dimensions, the class of polynomial-like maps is significantly larger than that of regular polynomial endomorphisms of $\mathbb{C}^k$ (i.e., those extending holomorphically to $\mathbb{P}^k$), see for instance \cite[Example 2.25]{DS10}.

\medskip

Every polynomial-like map $f$ 
gives a ramified covering from $U$ to $V$ and the \emph{topological degree} $d_t$ of $f$ is well-defined. We will always assume that 
we have
$d_t\ge 2$. The (compact) set $K:= \bigcap_{n=1}^{\infty} f^{-n}(U)$ is 
 called
the \emph{filled-in Julia set} of $f$.
It consists of 
the points whose orbit is well-defined. The system $(K,f)$ is a true dynamical system.
 \begin{definition}\label{def:degrees}
    Let $f:U\to V$ be a polynomial-like map. For every
    $0\le p\le k$ define
    \begin{equation}\label{stardegree}
    d_p=d_p(f):=\limsup_{n\to\infty}\sup_S\|(f^n)_*(S)\|_W^{1/n},
    \end{equation}
       where $W\Subset V$ 
       is an open
       neighbourhood of $K$ and the supremum in \eqref{stardegree} is taken over all positive closed $(k-p,k-p)$-currents whose mass is less than or equal to 1 on a fixed neighbourhood $W'\Subset V$ of $K$. We say
       that
       $d_p$ is 
       the
      \emph{dynamical degree of order $p$} of $f$.
  \end{definition}

   Recall that the mass of a positive $(p,p)$-current $S$ on the open set 
   $W$ is given by $||S||_W:=\int_WS \wedge \omega^{k-p}$, where $\omega$ is the standard K\"ahler form of $\mathbb C^k$.
 The definition above 
  is independent of $W,W'$
  \cite{DS03,DS10}.
  Moreover,
  we have
  $d_0=1$ and $d_k=d_t$.   In the case of endomorphisms of $\P^k$ of algebraic degree $d$, the above definitions reduce to $d_{p} = d^p$. 
    We say that $f$ has \textit{ dominant topological degree} if $d_{k-1}< d_t$.
By 
\cite[Theorem 1.3]{BDR23}, the
sequence
  $\{d_p\}_{0\leq p \leq k}$ is non-decreasing, hence we have 
    $\max_{0\leq p \leq k-1} d_p = d_{k-1} <d_t$  (and therefore also $d_t\ge 2$)
    for every polynomial-like map  with dominant topological degree. 

\medskip
Polynomial-like maps 
 with dominant
topological degree enjoy many of the dynamical properties of endomorphisms.
For instance, they admit a unique measure of maximal entropy $\log d_t$
\cite{DS03,DS10}. 
We will denote this measure by $\mu$ and define the \textit{Julia set} $J$ as the support of $\mu$. Observe that $J$ is a subset of the boundary of $K$.

\subsection{Speed of convergence}
\label{ss:speed-easy}
Let us fix a polynomial-like map $f:U\to V$
with topological degree $d_t\ge 2$. By 
\cite[Proposition 3.2.5]{DS03}
see also
\cite[Proposition 2.15]{DS10},
there exists a constant $0<\gamma<1$ such that
\begin{equation}\label{eq:ph}
|d_t^{-n} (f^n)_* \psi
-
\langle
\mu, \psi
\rangle
|
\lesssim \gamma^n
\end{equation}
for every function $\psi$ in a given compact family of pluriharmonic functions on $V$, where the implicit constant depends only on the family. 
 We will denote by $\gamma_0 = \gamma_0 (f)$ the infimum of the constants $\gamma$ for which \eqref{eq:ph} holds.
Assuming now that $f$ 
has  dominant topological degree, we 
set
\begin{equation}\label{eq:beta}
\beta = \beta(f):=\max\left\{{d_{k-1}},\gamma_0 d_t\right\}.
\end{equation}

In the case of endomorphisms of $\mathbb P^k$, 
every pluriharmonic function is constant. Hence, \eqref{eq:ph} is trivial and we can take $\beta= d_{k-1}= d^{k-1}$, where $d$ is the algebraic degree of the endomorphism. More generally, in a compact setting, one can take $\beta = d_{k-1}$. In our setting, 
we have the following weaker bound for $\beta$.
 \begin{lemma}\label{l:bound-beta}
  Let $f:U\to V$ be a polynomial-like map  with dominant topological degree $d_t$.
Then we have $\beta (f) \leq \sqrt{ d_{k-1} d_t}$.
\end{lemma}

\begin{proof}
   It is enough to show that $\gamma_0$ satisfies $\gamma_0 \leq \sqrt{d_{k-1}/d_t}$. Hence,
it suffices to show that
\eqref{eq:ph} holds for every $\gamma>\sqrt{d_{k-1}/d_t}$.

\medskip

   Observe that the map
$\psi \mapsto\sqrt{ \| i \partial \psi \wedge \bar \partial \psi \|_U}$
defines a norm on the space of pluriharmonic functions on $V$ with $\langle\mu, \psi\rangle=0$.
Moreover,
by the Cauchy-Schwarz 
 inequality, for any such function 
 we have
\[
0 \leq 
i \partial (d^{-n}_t (f^n)_* \psi ) \wedge \bar \partial (d^{-n}_t (f^n)_* \psi ) 
\leq 
d^{-n}_t (f^n)_* (i \partial \psi \wedge \bar \partial \psi).
\]
 As $i\partial \psi \wedge \bar \partial \psi$ is a positive closed $(1,1)$-current,
 the mass of the last term in the above expression satisfies
\[
\|
d^{-n}_t f^n_* (i \partial \psi \wedge \bar \partial \psi)\|_{U}
\lesssim d^{-n}_t (d_*)^n
\|
i \partial \psi \wedge \bar \partial \psi
\|_{U}
\quad \mbox{ as } n\to \infty
\]
for every $d_* > d_{k-1}$. The assertion follows.
\end{proof}

The following result, whose proof in a compact setting is essentially given in 
\cite[Lemme 4.2]{DS06}, gives the estimates for the rate in the convergence \eqref{e:intro-conv} that we will need.
We give the details 
of the proof for the reader's convenience.
A more precise version
in the setting of endomorphisms of $\mathbb P^k$ is given in \cite{DS10a}.

\begin{theorem}\label{t:speeda1}
    Let $f:U\to V$ be a polynomial-like map  with dominant topological degree $d_t$ and $\mu$
its equilibrium measure. 
 Let $\beta$ be as in \eqref{eq:beta} and 
take $\lambda$ such that
$\beta<\lambda<d_t$.
Then there exists a psh function $u_\lambda$
 with $u_\lambda <-1$
such that for every
$a\in U$,
$\psi\in \mathcal{C}^2(U)$, 
and
$n\in \mathbb N$
we have
\begin{equation}\label{eq:ulam}
    |\langle d_t^{-n}(f^n)^*\delta_a-\mu,\psi \rangle|\le A\|\psi\|_{\mathcal{C}^2(U)}
    \left(\frac{\lambda}{d_t}\right)^n 
      | u_\lambda(a)|,
   \end{equation}
where $A$ is a constant 
depending on $\lambda$
but independent of
$a$, $\psi$, and $n$.
\end{theorem}

\begin{proof}
By linearity, it is enough to show the assertion in the assumption that $\psi$ is psh and satisfies 
$\langle \mu, \psi\rangle=0$
and 
$0\leq dd^c \psi \leq dd^c \|z\|^2 = \omega$.
Define also the function $v_0:= \|z^2\|-
\langle\mu, \|z^2\|\rangle$, and observe
that it satisfies 
$\langle\mu, v_0\rangle=0$. Set also
\[v_n:= d_t^{-n} (f^n)_* v_0
\quad
\mbox{ and }
\quad
u_\lambda := \sum_{n=0}^\infty \left(\frac{d_t}{\lambda}\right)^{n} v_n - C_\lambda,\]
 where the constant $C_\lambda$ is 
chosen so
that
\[
u_\lambda < -1
\quad 
\quad
\mbox{ and }
\quad
\quad
\sum_{n\neq n_0} 
\left(\frac{d_t}{\lambda}\right)^{n} v_n
\leq C_\lambda
\quad 
\mbox{ for every } n_0 \in \mathbb N.
\]

Fix $\lambda'$ with $\beta < \lambda' < \lambda$.
As 
$v_n \lesssim (\lambda' / d_t)^n$ 
and 
$\|v_n\|_{L^p}\lesssim 
(\lambda' / d_t)^n$
for any $p\ge1 $
(see for instance
\cite[Theorem 2.33]{DS10}
and
\cite[Corollary 2.35]{DS10}),
we have that 
$u_\lambda$ is a well-defined (i.e., not identically equal to $-\infty$) psh function on $V$.
Setting $A_n (\psi) := \langle d_t^{-n}(f^n)^*\delta_a, \psi\rangle$,
we will show the inequality
\[|A_n (\psi)| \lesssim 
 \left(\frac{\lambda}{d_t}\right)^n 
    | u_\lambda(a)-1 |.
    \]
This will show the assertion, up to replacing $u_\lambda$
by
$u_\lambda -1$.

\medskip

By \cite[ Theorem 2.34]{DS10},
the definition of 
$\beta$, and the choice of
$\lambda$, for every
$\phi \in \mathcal C^2 (U)$
with $\langle\mu, \phi\rangle =0$
and $n\in \mathbb N$
we have
\[
  |\langle d_t^{-n}(f^n)^*\nu,\phi \rangle|\lesssim
  \| \phi\|_{\mathcal{C}^2(U)}
    \left(\frac{\lambda}{d_t}\right)^n
\]
for every \emph{smooth} probability measure $\nu$ compactly supported on $V$  (here the implicit constant can depend on $\nu$, but is independent of $\nu$ 
 if
this is
taken in a compact family of probability measures).
The mean inequality for psh functions implies that we have $A_n (\psi) \lesssim  (\lambda / {d_t})^n 
      $
      (where the implicit constant is now independent of $a$).
          In order to conclude,
      we need to prove a similar bound for $-A_n (\psi)$.

\medskip

It follows from the properties at the beginning of the proof that $v_0 - \psi$ is psh. Hence, we can 
write $\psi = v_0 - (v_0 - \psi)$
as a difference of psh functions, and we have
 \[- A_n (\psi) = -A_n (v_0) + A_n (v_0-\psi).\]
As an upper bound for 
$A_n (v_0-\psi)$ can be found with the same arguments as above, we only need to prove an upper bound for $-A_n (v_0) = - v_n (a)$.
By the definition of $u_\lambda$, we have 
$(d_t/\lambda)^{n} v_n \geq u_\lambda$
for every $n\in \mathbb N$.
It follows that we have 
\[
-A_n (v_0)
= - v_n (a)
\leq 
\left( \frac{\lambda}{d_t}\right)^n 
|u_\lambda 
|.
\]
The assertion follows.
\end{proof}

The following
immediate consequence of Theorem \ref{t:speeda1}
will be used 
to prove Theorem \ref{t:intro-nomasspluripolar}.
\begin{corollary}\label{t:speed}
Let $f:U\to V$ be a polynomial-like map  with dominant topological degree $d_t$ and $\mu$
its equilibrium measure.
 Let $\beta$ be as in \eqref{eq:beta}.
Fix an open set $U'$ with $K\subseteq U'\subseteq U$
and let $\omega_0$ be
a smooth probability measure 
on $U'$.
For any $\lambda$ with
$\beta<\lambda<d_t$
there exists a psh function $u_\lambda$  such that
 for every
 $\psi\in \mathcal{C}^2(U)$
and $n\in \mathbb N$
we have
\begin{equation}
|\langle d_{t}^{-n}(f^n)^*\omega_0-\mu,\psi \rangle|\le A\|\psi\|_{\mathcal{C}^2(U)} \left(\frac{\lambda}{d_t}\right)^{n}
\int_{U'}|u_\lambda|
\omega_0,
\end{equation}
where $A$ is a constant independent of $\omega_0$, $\psi$,
$U'$
and $n$. 
\end{corollary}

\section{Proof of Theorem \ref{t:intro-nomasspluripolar}}\label{s:4}

\subsection{Preparatory
lemmas}
In this section we prove 
 a couple of technical lemmas that we will need in the proof of Theorem \ref{t:intro-nomasspluripolar}.
  Theorem \ref{t:speeda1} and Corollary \ref{t:speed}
 are not used here.
 
 \medskip
 
 We fix a polynomial-like 
 map $f:U\to V$ 
 and
 the constant
 \begin{equation}\label{eq:defM}
     M
      = M(f):=\max_{1\le l\le k}\max_{z\in\overline{ f^{-1}(U)}} \|\nabla f_l(z)\|,
 \end{equation}
where the $f_l$'s
denote the components of $f$.
Recall that we denote by $\omega$ 
the standard Kahler form on $\mathbb C^k$.
For every $m\in \mathbb N$ and $1\leq i,j\leq k$, we also
set
\[\alpha^{(m)}_{i,j} :=
\sum_{l=1}^k
\frac{\partial f_l^m}{\partial z_i} \frac{\partial \bar{f}_{l}^m}{\partial \bar{z}_j},
\]
 where the $f^m_l$'s denote
the components of $f^m$.
Observe that the $\alpha^{(m)}_{i,j}$'s
are smooth functions on $f^{-m} (V)$ and
for every $ N\geq m$
we have 
\[(f^m)^*\left.\omega\right|_{f^{- N}(U)}=\sum_{1\le i,j\le k}\alpha^{(m)}_{i,j}dz_i\wedge d\bar{z}_j.
\]

\begin{lemma}\label{l:estdiff}
 For every 
 $1\le m\le N-1$, we have
 \[|\alpha^{(m)}_{i,j}|_{f^{-N}(U)}\le k^{2m-1}M^{2m}.\]
\end{lemma}
\begin{proof}
 We claim that for 
 every
 $z\in \overline{f^{-N}(U)}$,
 every
 $1\le m\le N-1$, and every $1\leq i,j\leq k$
 we have
\begin{equation}\label{eq:estdiff}
    \left|\frac{\partial f_j^m(z)}{\partial z_i}\right|\le k^{m-1} M^m.
\end{equation}
For $m=1$,  \eqref{eq:estdiff} 
follows from the definition of 
$M$. 
Assume now that   \eqref{eq:estdiff}  holds for
$m-1\ge 0$ instead of $m$.
As $f^{m-N}(U)\Subset f^{-1}(U)$,
we have $\left|\frac{\partial f_j}{\partial z_l}\left(f^{m-1}(z)\right)\right|\le M$ for every $1\leq j,l\leq k$.
So, for every 
$1\le i,j\le k$,
we have
\begin{align*}
 \left|\frac{\partial f_j^m(z)}{\partial z_i}\right|
 &=\left|\sum_{l=1}^k \frac{\partial f_j}{\partial z_l}\left(f^{m-1}(z)\right)\frac{\partial f_l^{m-1}(z)}{\partial z_i}\right|
\le\sum_{l=1}^k \left|\frac{\partial f_j}{\partial z_l}\left(f^{m-1}(z)\right)\right|\left|\frac{\partial f_l^{m-1}(z)}{\partial z_i}\right| \\
&\le\sum_{l=1}^k M \cdot k^{m-2} M^{m-1}
 = k^{m-1} M^m
\end{align*}
on $\overline{f^{-N}(U)}$.
Hence, 
\eqref{eq:estdiff} holds for all $1\leq m\leq N-1$. It follows that we have
\begin{align*}
    |\alpha^{(m)}_{i,j}(z)|=\left|\sum_{l=1}^k\frac{\partial f_l^m}{\partial z_i} \frac{\partial \bar{f}_{l}^m}{\partial \bar{z}_j}\right|\le \sum_{l=1}^k\left|\frac{\partial f_l^m}{\partial z_i}\right| \left|\frac{\partial \bar{f}_{l}^m}{\partial \bar{z}_j}\right|\le k^{2m-1}M^{2m},
\end{align*}
as desired.
\end{proof}

From now on, for simplicity, given $m_1\leq \dots\leq m_k\in\mathbb N$,
we will use the notation
\begin{equation}\label{e:notation-Omega}
\Omega_{m_1, \dots, m_k}
:=
({f^{m_1}})^*\omega\wedge \dots\wedge({f^{m_{k}}})^*\omega.
\end{equation}
Observe that $\Omega_{m_1, \dots, m_k}$
is a smooth volume form on $f^{-m_k} (V)$.
We will also denote by 
$\varphi_{m_1, \dots, m_k}$ the Radon-Nikodym density of $\Omega_{m_1, \dots, m_k}$
with respect to $\omega^k$. This is a positive
smooth function on $f^{-m_k} (V)$.
The following bound 
for ${\varphi}_{m_1,\dots, m_k}$
immediately follows from Lemma \ref{l:estdiff}.

\begin{corollary}\label{c:phi}
There exists a constant $C$ 
such that, for every
 $0\le m_1\le...\le m_k\le n\leq
 N-1$ 
we have
\[\varphi_{m_1, \dots, m_k}\le C {(kM)}^{2kn} 
\quad 
\mbox{ on }
\quad
{f^{-N}(U)}.\]
\end{corollary}

Observe
that
$\Omega_{m_1, \dots, m_k}$
as above
satisfies 
\[
\int_{f^{-m_k} (V)} \Omega_{m_1, \dots, m_k} \lesssim d_t^{m_k},
\]
where the implicit constant is independent of $m_1, \dots, m_k$,
see for instance \cite{DS03,DS10}.
In the following, 
we will 
need the following better bound for the
integral above in the case where $m_1=1$.

\begin{lemma}\label{l:intlessdegree}
 There exists a function $\eta:\mathbb N\to \R$
 with
 $\limsup_{n\to\infty}\eta(n)^{1/n}=1$
 such that
 \begin{equation}\label{eq:ledk-1}
   \int_{ f^{-N}(V)} 
   \Omega_{0,  m'_1, \dots, m'_{k-1}}
   \le \eta(N) (d_{k-1})^{N}   
 \end{equation}
for every $1\le  m'_1\leq \dots
\leq m'_{k-1}\le N-1$.
\end{lemma}
\begin{proof}
   Let $X$ be an analytic subset of $V$ of pure dimension 
   $k-1$.
   There exists a function  $ \eta : \mathbb N\to \mathbb R$ with $\limsup_{n\to\infty} \eta(n)^{1/n}=1$ and depending only from
   the mass of $[X]$
   such that
\begin{equation}\label{Xdeltap}
    \int_{{f^{-N}(V)}}[X]
    \wedge 
        (f^{m'_1})^*\omega\wedge
    \ldots\wedge (f^{m'_{k-1}})^*\omega
    \le \eta(N)\,{(d_{k-1})^N}
\end{equation}
for every $1\leq
 m'_1\leq \dots \leq m'_{k-1}\leq N-1$.
The proof of \eqref{Xdeltap}
follows the strategy used by Gromov
to estimate
the topological entropy of endomorphisms of
$\mathbb P^k$, see for instance \cite{G03},
and adapted by Dinh and Sibony
\cite{DS03,DS10}
to the setting of
polynomial-like maps.
Since only minor modifications are needed, we
refer to
\cite[Lemma A.2.6]{B16} for a complete proof.
The inequality \eqref{eq:ledk-1} is deduced from \eqref{Xdeltap}
(by possibly multiplying the function $\eta(n)$ by a bounded factor)
as $\omega$ can be written as an
average of currents of integration on $(k-1)$-dimensional
analytic sets.
\end{proof}

 Finally, we will need the following lemma
about the integrals of psh functions.

\begin{lemma}\label{l:limitmeasure}
 Let $u$ be a psh function on $V$. Then, 
there exists a 
positive  constant $A_1$ depending on $u$
 such that 
for every
$0\le m_1\le...\le m_k\leq N\in \mathbb N$ 
we have
\[
\int_{f^{-N}(U)}|u|
 \Omega_{m_1, \dots, m_k}
  \le 
 A_1
  +
  m_k^2
  \int_{f^{-N}(U)}
 \Omega_{m_1, \dots, m_k}.
\]
\end{lemma}

\begin{proof}
For every $n\leq N\in \mathbb N$,
set
\[W_{n,N}:=f^{-N}(U)\cap \{u<-n^2\}.\]
Then, for all $m_1, \dots, m_k,N$ 
as in the statement and $m_k \leq n \leq N$,
we have
\[
\begin{aligned}
  \int_{f^{-N}(U)}
|u
    | 
 \Omega_{m_1, \dots, m_k}
 & \le \int_{f^{-N}(U)
\setminus W_{n,N}
}
|u
    | 
\Omega_{m_1, \dots, m_k}
  +  \int_{W_{n,N}}|u
    | 
  \Omega_{m_1, \dots, m_k} 
  \\
  &\le  n^2
 \int_{f^{-N}(U)}
 \Omega_{m_1, \dots, m_k}
  + \int_{W_{n,N}}|u
    |  {\varphi_{m_1, \dots, m_k}\omega^k}.
    \end{aligned}\]
Hence, it is enough to show that the last integral in the above expression
is bounded uniformly in $m_1, \dots, m_k, n$, and $N$ (we will actually show that it tends to $0$ 
if $n\to \infty$).

\medskip

 As we have $W_{n,N}\subset \{u< -n^2\}$, by the Skoda estimates for psh functions \cite{S82}
there exists  $\alpha>0$ and $n_0\in\mathbb{N}$ such that for any $n>n_0$ 
we have
\begin{equation}\label{eq:len3}
  \int_{W_{n,N}}|u
    |   \omega^k\le  C_1  e^{-\alpha n^2},
  \end{equation}
 for some positive
 constant
 $C_1$ independent from $n$.
By Corollary
\ref{c:phi},
there exist two
constants
 $ C_2$ and $M$, independent from $m_1, \dots, m_k$, $n$, and $N$  
 such that $\varphi_{m_1, \dots, m_k}\le  C_2 (k M)^{2kn}$
 on $f^{-n} (U)\supseteq f^{-N} (U)$.
We deduce from this estimate and
 \eqref{eq:len3} that we have
\[
\int_{W_{n,N}}
|u
    | 
\varphi_{m_1, \dots, m_k} \omega^k 
 \le   C_2 (kM)^{2kn} \int_{W_{n,N}}|u
    |   \omega^k 
    \le  C_1C_2 (kM)^{2kn}  e^{-\alpha n^2} 
    \]
for every $n > n_0$.
Choosing 
$A_1:= \max_{n\in\mathbb N}
 C_1C_2 (kM)^{2kn}  e^{-\alpha n^2}$   completes the proof.
\end{proof} 
\subsection{Proof of Theorem \ref{t:intro-nomasspluripolar}} 
We can now 
prove the 
following statement, which,
 by Lemma \ref{l:bound-beta},
gives
a more precise version of Theorem \ref{t:intro-nomasspluripolar}.

\begin{theorem}\label{t:logp-nomasspluripolar}
  Let $f:U\to V$ 
be a polynomial-like map  with dominant
topological degree
and
$\beta$ be as in \eqref{eq:beta}.
Then, every 
   ergodic measure
   $\nu$
   whose measure-theoretic entropy
  satisfies 
   $h_\nu(f)>\log \beta$
  is supported on the Julia set $J$. 
   \end{theorem}

\begin{proof}
 Fix  $\nu$
as in the statement and  $\lambda$ with $\beta<\lambda<  e^{h_\nu(f)} \leq d_t$.
Let $F\subset U\setminus J$ be a closed set.
We
are going to show that we have
$h_t (f, F)\le \log{\lambda}$, where $h_t(f,F)$ denotes the topological entropy of $f$ on $F$. By the 
relative
variational principle, this implies 
$\nu (F)=0$, and hence that the support of $\nu$ is contained in $J$, as desired.
Observe that we can assume, with no loss of generality, that we have $F\subset K$.

\medskip

Let $W$ be an open neighbourhood of $F$ with $\overline{W}\cap J=\emptyset$. 
 Using Gromov's contruction \cite{G03}  and 
the same arguments as in 
the proof of \cite[Theorem 1.108]{DS10} (see also the proof of \cite[Lemma A.2.6]{B16}),
we have
$$h_t(f, F)=h_t(f,F\cap K)\le \mathrm{lov} (f,W):=\limsup_{{ N}\to\infty}\frac{1}{N}\log\mathrm{vol}(\Gamma_N^W),$$
where, for every $N\in \mathbb N$,
$\Gamma_N^W$ denotes the subset of ${U}^{N+1}$ given by
$$ \Gamma_N^W:=\{(z,f(z),
\dots,
f^{N-1}(z)),z\in W\cap f^{-N}(U)\}.$$
Note that, 
for every $N\in \mathbb N$, we have
$$\mathrm{vol}(\Gamma_N^W)={\sum_{\substack{0 \le n_i \le N-1, \\ 1 \le i \le k}}} \int_{W\cap f^{-N}(U)} 
(f^{n_1})^*\omega\wedge\dots\wedge (f^{n_k})^*\omega.
$$

As the number of terms in the sum is polynomial in $N$,
it is enough to consider separately
each term of the
sum on the right-hand side of the above expression. Hence, without loss
of generality, we can assume that we have
$0\leq n_1\le \dots \leq n_k\leq N-1$
and we need to show the inequality
\[
\limsup_{N\to \infty}
\frac{1}{N}
\log \int_{W\cap f^{-N}(U)} (f^{n_1})^*\omega\wedge...\wedge (f^{n_k})^*\omega
\le \log {\lambda}.
\]

Recall the notation
\eqref{e:notation-Omega}. 
We will now
consider the smooth form
\[
\Omega'_{n_1, \dots, n_k}
:=
\Omega_{0,n_2-n_1, \dots n_k-n_1}
=
\omega\wedge
 (f^{n_2-n_1})^*\omega \wedge
\ldots
\wedge (f^{n_k-n_1})^*\omega.
\]
Observe that we have
\[\Omega_{n_1, \dots, n_k}= (f^{n_1})^* (\Omega'_{n_1, \dots, n_k}).\]
Fix also
an
open set $\tilde W$ with $\tilde W \Supset W$ and $\tilde W \cap J=\emptyset$
and a
smooth function $0\le\psi\le 1$ 
with compact support in $\tilde W$ and 
such that $\psi|_W=1$.
Since $\mu|_{\tilde{W}}=0$,
by Corollary
\ref{t:speed}
applied with 
\[U' = f^{- (N-n_1)} (U)
\quad
\mbox{ and
} \quad
\omega_0=
\|\Omega'_{n_1, \dots, n_k}\|_{f^{-(N-n_1)}(U)}^{-1}
\cdot
(\Omega'_{n_1, \dots, n_k})_{|{f^{-(N-n_1)}(U)}}
\]
 there exists  a psh function $u_\lambda$ such that 
\begin{align*}
       \int_{ f^{-N}(U)}
       \psi
       \Omega_{n_1, \dots, n_k}
              =
              &\left|\left\langle (f^{n_1})^*
       ({\Omega'_{n_1,\dots,n_k}})
       -d_t^{n_1}
\|\Omega'_{n_1, \dots, n_k}\|
       _{f^{-(N-n_1)}(U)}
              \mu,\psi \right\rangle\right| \\
       \le& A\|\psi\|_{ \mathcal C^2 (U)}
        \lambda^{n_1}
       \int_{ f^{-(N-n_1)}(U)} {|u_\lambda|}
       \Omega'_{n_1, \dots, n_k},
\end{align*}
where
$A>0$ is a constant independent of $n_1, \dots, n_k,N$, 
and $\psi$.
We deduce from the above inequality
and
Lemma \ref{l:limitmeasure} 
(applied
 with $m_j=n_j-n_1$
 for all $j$,
so that $\Omega_{m_1, \dots, m_k}= \Omega'_{n_1,\dots, n_k}$)
that
there exists
a positive
constant $A_1$ (depending on {$\lambda$}, but independent of $n_1, \dots, n_k$, $N$, and $\psi$)
such that
 \begin{equation*}
     \int_{ f^{-N}(U)}\psi
         \Omega_{n_1, \dots, n_k}
     \le  
   A
         \|\psi\|_{ \mathcal C^2 (U)}
     {\lambda}^{n_1}
\Big(   A_1 +
  {N^2}
    \int_{f^{-N+n_1}(U)}
    \Omega'_{n_1,\dots, n_k}\Big).
    \end{equation*}
Recalling the definitions of 
$\Omega_{n_1, \dots, n_k}$
and $\psi$, we deduce from 
 the above expression
that we have
 \begin{equation}\label{e:omega-steps}
      \int_{W\cap f^{-N}(U)} (f^{n_1})^*\omega\wedge\dots
        \wedge (f^{n_k})^*\omega
      \le 
       A \|\psi\|_{\mathcal C^2 (U)}
      {\lambda}^{n_1}
       \Big( A_1
      +
            { N^2}
      \int_{ f^{-N+n_1}(U)}
 \Omega'_{n_1,\dots, n_k}\Big).
  \end{equation}
 By Lemma \ref{l:intlessdegree}
 (applied with $m'_j =n_j-n_1$)
 we have 
 \begin{equation}\label{eq:intled}
     \int_{ f^{-N+n_1}(U)} 
    \Omega'_{n_1, \dots, n_k}
    =
    \int_{f^{-N+n_1} (U)}
\Omega_{0, n_2-n_1, \dots, n_k-n_1}    
    \le \eta(N)
        (d_{k-1})^{N-n_1},
 \end{equation}
 where the function $\eta$
 satisfies
$\lim_{n\to\infty} \eta(n)^{1/n}=1$. Combining \eqref{e:omega-steps} and \eqref{eq:intled}, we obtain
\[\begin{aligned}
 \int_{W\cap f^{-N}(U)} (f^{n_1})^*\omega\wedge\dots\wedge (f^{n_k})^*\omega
 & \le
 A\|\psi\|_{\mathcal{C}^2(U)}
 {\lambda}^{n_1} 
 \big(
 A_1 +  
 { N^2}
 \eta (N)(d_{k-1})^{N-n_1}\big)\\
 & \le  \tilde{\eta}(N)
  {\lambda}^{n_1}
 (d_{k-1})^{N-n_1}\\
 & \leq 
 \tilde{\eta}(N)
  {\lambda}^{N},
  \end{aligned}\]
where the function $\tilde \eta$ 
(which can depend on $\lambda$)
satisfies 
$\lim_{n\to\infty} \tilde{\eta}(n)^{1/n}=1$
and in the last step we used the inequality
$d_{k-1} \le \lambda$.
Consequently, we have
$$\limsup_{N\to\infty}\frac{1}{N}\log\mathrm{vol}(\Gamma_N^W)\le
\limsup_{N\to \infty}
\frac{1}{N}\log \left(
N^{k} \tilde{\eta}(N)
 {\lambda}^N
\right)
\le
\log {\lambda},$$
which gives
$ h_t(f,F)\le \log {\lambda}$, as desired.
This concludes the proof.
\end{proof}

\section{Further results and remarks}\label{s:5}

\subsection{Hausdorff dimension of the Julia set} 
In this section we fix a polynomial-like map $f : U\to V$ 
 with dominant topological degree. Let $\nu$ be an  ergodic 
 probability
measure with 
$h_\nu(f)>\log {\beta}$
 (where $\beta$ is as in \eqref{eq:beta})
and denote by
$$0< L_k(\nu)\le L_{k-1}(\nu)\le\ldots\le L_1(\nu)$$
its
Lyapunov exponents, counting multiplicities. The inequality $0<L_k(\nu)$ is proved in  \cite[Theorem 4.1]{BR22} (see \cite{dTh08,D12}
for the case of endomorphisms).
We also have the following property, 
originally proved by Dupont for endomorphisms  \cite[Theorem A]{D11}.
As the proof is local, it also applies in our setting.

\begin{theorem}\label{t:infdelta}
   Let $f$, $\nu$,
   and $L_j, 1\le j\le k$ be as above. Then for $\nu$-almost all $z\in V$ we have
$$\liminf_{r\to0}\frac{\log\nu(B(z,r))}{\log r}\ge  \frac{\log {d_{k-1}}}{L_1(\nu)}+\frac{h_{\nu}(f)-\log {d_{k-1}}}{L_k(\nu)},$$
 where $B(z,r)$ is the ball
 {of} 
 radius $r>0$ and centred at $z$.   
In particular, 
for every 
Borel set $E\subset V$
with $\nu(E)>0$,
the Hausdorff dimension 
 $\mathrm{dim}_{H}E$ 
of $E$ satisfies
$$  \mathrm{dim}_{H}E
\ge  \frac{\log {d_{k-1}}}{L_1(\nu)}+\frac{h_{\nu}(f)-\log {d_{k-1}}}{L_k(\nu)}. $$
  \end{theorem}

As in \cite{D11}, the following consequence
of Theorem \ref{t:infdelta} gives a lower bound 
for
the Hausdorff dimension of the Julia set $J$ of $f$.

\begin{corollary}
  Let $f$ be as above
   and
  $\beta$ be as in \eqref{eq:beta}.
Then, for every ergodic
   measure $\nu$
   whose measure-theoretic entropy
  satisfies 
  $h_\nu(f)>\log {\beta}$,  we have
  $$ \mathrm{dim}_{\mathcal{H}}J\ge \frac{\log { d_{k-1}}}{L_1(\nu)}+\frac{h_{\nu}(f)-\log { d_{k-1}}}{L_k(\nu)}.$$
\end{corollary}
\begin{proof}
Theorem
 \ref{t:logp-nomasspluripolar}
implies that $\nu$ is supported on  $J$, and hence 
$\nu(J)=1$. Therefore, the assertion follows from 
Theorem
\ref{t:infdelta}. 
\end{proof}

\subsection{Strong stability in families of polynomial-like maps}

We consider in this section a holomorphic family of polynomial-like maps $(f_\tau)_{\tau \in M}$
 with dominant topological degree
parametrized by a complex manifold $M$,
see for instance \cite[Section 3.8]{DS03} and \cite[Section 2.5]{DS10}.
Recall that all the 
$f_\tau$'s
have the same topological degree $d_t$, which we assume to be at least 2,
 and that the map $\tau \mapsto d_{k-1}(f_\tau)$ is upper semicontinuous. 
We will denote by 
$\mu_\tau$
the equilibrium measure of $f_\tau$.
The dynamical stability for such families has been studied in \cite{B19}, as a generalization of the theory developed for families of endomorphisms of $\mathbb P^k$ in \cite{BBD18}, see also \cite{Pha05} and \cite{dM03, Ly83a, MSS83} for the case $k=1$.
Following \cite{BBD18,B19}, let us denote by $\mathcal{J}$  the set of all holomorphic maps $\gamma : M \to \mathbb C^k$ such that $\gamma(\tau)$ belongs to the Julia set of $f_\tau$ for all $\tau \in M$. Define $\mathcal{F}:\mathcal{J}\to \mathcal{J}$ as $\mathcal F \gamma (\tau) := f_\tau (\gamma(\tau))$.

\begin{definition}\label{def:lamination}
A \emph{dynamical lamination}
for the family
$(f_\tau)_{\tau \in M}$
is an $\mathcal{F}$-invariant subset $\mathcal{L}$
of $\mathcal{J}$ such that
\begin{enumerate}
    \item $\Gamma_\gamma\cap \Gamma_{\gamma'}=\emptyset$ for every
    $\gamma\ne \gamma'\in\mathcal{L}$, where $\Gamma_\gamma$ is the graph of $\gamma$
    in $M\times \C^k$;
    \item $\Gamma_\gamma\cap GO(C_f)=\emptyset$
for every    
$\gamma\in\mathcal{L}$, where $C_f$ is the critical set of
the map $f: (\tau,z)\mapsto (\tau, f_\tau (z))$,
and $GO (f) := \cup_{n,m\geq 0} f^{-m} ( f^{n} (C_f))$;
    \item $\mathcal{F}:\mathcal{L}\to \mathcal{L}$ is $d_t$-to-$1$.
\end{enumerate}
\end{definition}

The dynamical stability of the family $(f_\tau)_{\tau\in M}$ is defined and characterized in \cite{BBD18,B19} by a number of equivalent conditions,
among which there is 
the existence of 
a dynamical lamination $\mathcal{L}$ such that 
 $\mu_\tau(\{\gamma(\tau):\gamma\in\mathcal{L}\})=1$
    for all $\tau\in M$.
It was proved in \cite{BR22} that stability implies (and is then equivalent to) the existence of a dynamical lamination associated to any ergodic
measure 
 for some
$f_{\tau_0}$
whose entropy 
in larger than $\log d_{k-1}  (f_{\tau_0})$ 
and
supported on the Julia set.
The following
is another
corollary of
 our main results,
which permits to 
remove
the  assumption on the support of the measure in \cite[Section 4.3]{BR22}
when the measures
satisfy the 
stronger bound on their measure-theoretic entropy as in
Theorem
 \ref{t:logp-nomasspluripolar}.

\begin{corollary}\label{c:stability}
    Let $M$ be a connected and simply connected
 complex manifold and 
 $(f_\tau)_{\tau\in M}$
 a 
 stable
family of polynomial-like maps  with dominant
topological
degree. Fix $\tau_0\in M$
and let $\beta(f_{\tau_0})$ be as in \eqref{eq:beta}. 
Then, there exists a dynamical lamination  $\mathcal{L}$ such that 
 $\nu(\{\gamma(\tau_0):\gamma\in\mathcal{L}\})=1$
    for every
    ergodic
    $f_{\tau_0}$-invariant  probability measure
    $\nu$ with $h_{\nu}(f_{\tau_0})>\log  \beta (f_{\tau_0})$. 
\end{corollary}
\begin{proof}
 By Theorem
 \ref{t:logp-nomasspluripolar},
every
ergodic 
$f_{\tau_0}$-invariant probability measure $\nu$ with $h_{\nu}(f_{\tau_0})>\log { \beta}(f_{\tau_0})$ is supported in the Julia set $J_{\tau_0}$ of $f_{\tau_0}$. Thus, the conclusion follows from \cite[Corollary 4.5]{BR22}.
\end{proof}

\end{document}